\documentclass[12pt]{amsart}
\usepackage{amssymb}
\usepackage{mathrsfs}
\usepackage{bm}
\usepackage{graphicx}
\usepackage{color}
\usepackage[all]{xy}

\allowdisplaybreaks[4]

\newtheorem{THEOREM}{Theorem}
\newtheorem{thm}{Theorem}[section]
\newtheorem{lem}[thm]{Lemma}
\newtheorem{cor}[thm]{Corollary}
\newtheorem{prop}[thm]{Proposition}
\newtheorem{defn}[thm]{Definition}
\newtheorem{rmk}[thm]{Remark}

\numberwithin{equation}{thm}

\newenvironment{red}
{\relax\color{red}}
{\hspace*{.5ex}\relax}

\newcommand{\ber}{\begin{red}}
\newcommand{\er}{\end{red}}

\newcommand\Z{\mathbb{Z}}
\newcommand\N{\mathbb{N}}
\newcommand\F{\mathbb{F}}

\newcommand\Hom{\operatorname{Hom}}
\newcommand\End{\operatorname{End}}

\newcommand\Span{\operatorname{span}}

\newcommand\Rad{\operatorname{Rad}}
\newcommand\Soc{\operatorname{Soc}}

\newcommand\GL{\operatorname{GL}}

\newcommand{\isom}{\,\raise2pt\hbox{$\underrightarrow{\sim}$}\,}

\begin{document}

\setlength{\baselineskip}{4.9mm}
\setlength{\abovedisplayskip}{4.5mm}
\setlength{\belowdisplayskip}{4.5mm}


\renewcommand{\theenumi}{\roman{enumi}}
\renewcommand{\labelenumi}{(\theenumi)}
\renewcommand{\thefootnote}{\fnsymbol{footnote}}
\renewcommand{\thefootnote}{\fnsymbol{footnote}}
\parindent=20pt


\setcounter{section}{0}




\title{Tame block algebras of Hecke algebras of classical type}
\author{Susumu Ariki}

\address{S.Ariki : Department of Pure and Applied Mathematics, Graduate School of Information Science and Technology, Osaka University, 1-5 Yamadaoka, Suita, Osaka 565-0871, Japan}
\subjclass[2010]{20C08, 20C33, 16G20, 16G60}

\maketitle 

\begin{center}
{\it On the occasion of Professor Richard Dipper's retirement}
\end{center}

\begin{abstract}
We classify tame block algebras of Hecke algebras of classical type over an algebraically closed field of characteristic not equal to two. 
\end{abstract}

\section{Introduction}

Hecke algebras associated with finite Weyl groups have been studied intensively in the past several decades because of its importance in Lie theory. 
In the modular representation theory of finite groups of Lie type over algebraically closed fields of non-defining characteristic, 
they appear as the endomorphism algebras of the modules which are Harish-Chandra induced from the cuspidal modules of Harish-Chandra series. 
Utilizing the modular representation theory of Hecke algebras of type $A$ developed in \cite{DJ86(b)} and \cite{DJ87}, 
Richard Dipper in the papers \cite{D85I}, \cite{D85II}, and Gordon James in the paper \cite{J86}, 
gave the classification of irreducible modules of $\GL_n(q)$ in the non-defining characteristic case. 
The bijection between the two labels was established in \cite{DJ86(a)}. Then, they introduced the $q$-Schur algebra, which is an algebra defined from 
the Hecke algebras of type $A$, and showed that the modular representation theory of the $q$-Schur algebra knows the decomposition numbers of 
$\GL_n(q)$ in the non-defining characteristic case \cite{DJ89}. The relationship between the module category of the $q$-Schur algebra and 
the module category of the finite general linear group in non-defining characteristics is given in \cite{BDK01} via the cuspidal algebras.

This success motivated Dipper and James to study the modular representation theory of Hecke algebras in other types and 
studied Hecke algebras of type $B$ in \cite{DJ92} and \cite{DJM95}. However, the study of the modular representation theory of Hecke algebras of type $B$ 
required various new ideas, and the Lascoux-Leclerc-Thibon conjecture allowed the author to contribute the later development in \cite{Ar96}, \cite{AM00}, \cite{Ar01(b)} etc.
See survey papers \cite{Ar01(a)} and \cite{Ar08}.
\footnote{It was expected that the James' conjecture together with the solution of the Lascoux-Leclerc-Thibon conjecture would allow us 
to give formulas for the decomposition numbers in certain stable region of parameters. However, due to
counterexamples by Geordie Williamson, we cannot expect any reasonable 
answer to the decomposition number problem at this moment.}
Nevertheless, the modular representation theory of Hecke algebras itself is still far from well-understood, 
and little has been done to explore relationship between subquotient categories of the module category or the derived category of a finite group of Lie type and those of Hecke algebras, neither. 

In this paper, we consider tame block algebras of Hecke algebras of classical type. Recall that Drozd's dichotomy theorem tells us that 
we have to choose among the stages either
\begin{itemize}
\item[(a)]
studying representations over algebras of tame representation type, or
\item[(b)]
finding results on Grothendieck group level such as character formulas, or studying relationship between various subquotient categories using 
modules with good properties,
\end{itemize}
because we cannot expect detailed study of the module categories for algebras of wild representation type. 

Our recent results \cite[Theorem A, B]{Ar17} give criteria to tell the representation type of block algebras of Hecke algebras of classical type and 
the purpose of this article is to work more to determine 
Morita equivalence classes of tame block algebras, so that we have settled the stage (a) for Hecke algebras of classical type in principle, 
and answer the decomposition number problem for tame block algebras as a corollary. 

We note that although the definition of the Hecke algebras of classical type is very simple, 
the proof of Theorem A and Theorem B requires combination of various results in the development of the theory of cyclotomic quiver Hecke algebras, which are also called cyclotomic Khovanov-Lauda-Rouquier algebras: results by Brundan-Kleshchev \cite{BK09(a)}, \cite{BK09(b)}, Chuang-Rouquier \cite{CR08}, Kang-Kashiwara \cite{KK11}, \cite{Kash11}, together with classical results by Rickard \cite{R89(b)} and Krause \cite{Kr97}. 

For finite representation type, we have already proved the following Theorem C. For the proof, the cellularity plays an important role, and block algebras of Hecke algebras are known to be cellular by the old results of Dipper, James and Murphy I have mentioned above and by Geck's result \cite{G07}. 
Because of the cellularity, we may also speak of decomposition numbers, and if we know the decomposition numbers, we may give dimension formulas for 
irreducible modules. 

\begin{THEOREM}[{\cite[Theorem C]{Ar17}}]\label{Thm1}
Suppose that $B$ is a block algebra of Hecke algebras of classical type over an algebraically closed field of characteristic not equal to two. If $B$ is of finite representation type, then 
$B$ is a Brauer line algebra, that is, a Brauer tree algebra whose Brauer tree is a straight line without exceptional vertex. In particular, the decomposition matrix is of the 
following form. 
\[
\begin{pmatrix}
1 &    &            & \\
1 & 1 &             & \\
   & 1 &  \ddots & \\
   &    & \ddots  & 1 \\
   &    &             & 1
\end{pmatrix}
\]
\end{THEOREM}

For tame block algebras of Hecke algebras of classical type over an algebraically closed field of characteristic not equal to two, we classify their
Morita equivalence classes in this paper. This has become possible by the confirmation  in \cite{Ar17} of the author's conjecture that tame block algebras are 
Brauer graph algebras.\footnote{We expect that this remains true for wider class of cyclotomic Hecke algebras, or cyclotomic quiver Hecke algebras.} 
Then, applying recent results from the representation theory of Brauer graph algebras, we obtain the following result. 
Theorem \ref{Thm2} says that even though there are infinitely many tame block algebras, their Morita equivalence classes are very restricted. 

\begin{THEOREM}\label{Thm2}
Suppose that $B$ is a block algebra of Hecke algebras of classical type over an algebraically closed field of characterisitc not equal to two. 
If $B$ is of tame representation type and not of finite representation type, then 
$B$ is Morita equivalent to one of the algebras below.
\begin{itemize}
\item[(1)]
For Hecke algebras of type $A$ and type $D$, Brauer graph algebras whose Brauer graph are one of the following.
\[
\begin{xy}
(0,0) *[o]+[Fo]{2}="A", (10,0) *[o]+[Fo]{\hphantom{2}}="B", (20,0) *[o]+[Fo]{2}="C",
\ar@{-} "A";"B"
\ar@{-} "B";"C"
\end{xy}\quad\quad
\begin{xy}
(0,0) *[o]+[Fo]{2}="A", (10,0) *[o]+[Fo]{2}="B", (20,0) *[o]+[Fo]{\hphantom{2}}="C",
\ar@{-} "A";"B"
\ar@{-} "B";"C"
\end{xy}
\]
They occur only when the quantum characteristic is $e=2$.
\item[(2)]
For Hecke algebras of type $B$ with two parameters, either 
\begin{itemize}
\item[(a)]
the Brauer graph algebras in (1), or the symmetric Kronecker algebra, which is the Brauer graph algebra with one non-exceptional vertex and one loop,  
if  the quantum characteristic $e=2$, or
\item[(b)]
the Brauer graph algebra whose Brauer graph is 
\[
\begin{xy}
(0,0) *[o]+[Fo]{2}="A", (10,0) *[o]+[Fo]{2}="B", (20,0) *[o]+[Fo]{2}="C",
\ar@{-} "A";"B"
\ar@{-} "B";"C"
\end{xy}
\]
if the quantum characteristic is $e\ge4$ and $Q=-1$.
\end{itemize}
\end{itemize}
\end{THEOREM}

To prove Theorem \ref{Thm2} for type $A$ and type $B$, we use the silting theory initiated in \cite{AI12}. Then, the result is obtained by simple application of 
recent development by Takuma Aihara and his collaborators\footnote{The author is grateful to Dr. Ryoichi Kase for drawing his attention to Aihara's work. }
in \cite{AAC15}, \cite{A13}, \cite{A15} and \cite{AM17}, 
because the tame block algebras are derived equivalent to Brauer graph algebras by \cite[Theorem A, B]{Ar17}. 
We note that Brauer graph algebras are symmetric special biserial algebras and the class of Brauer graph algebras is closed under 
derived equivalence if the ground field is algebraically closed of characterisitc not equal to two by \cite{AZ17}. 
For Hecke algebras of type $D$, we need a little more extra argument to obtain the result. We embed Hecke algebras of type $D$ to 
Hecke algebras of type $B$ with a special choice of two parameters and 
use Specht module theory in the language of Kashiwara crystal to control the branching rule and prove that 
irreducible modules remain irreducible under the restriction from the Hecke algebras of type $B$ to type $D$. 

As a consequence of Theorem \ref{Thm2}, we can determine the decomposition numbers for tame block algebras. 
The result also shows that the Morita classes in the derived equivalence classes of the tame block algebras all appear as tame block algebras again. 
As one can expect, this is no more true for wild block algebras. We give an example in the last section. 

\section{Preliminaries}

Throughout the paper, $K$ is an algebraically closed field of characterisitc not equal to two. The \textbf{Hecke algebra of type $A$} is the $K$-algebra $H^A_n(q)$, where $1\ne q\in K^\times$,  
defined by generators $T_1,\dots,T_{n-1}$ and relations
\begin{gather*}
(T_i-q)(T_i+1)=0\;\;(1\le i\le n-1),\quad  T_iT_j=T_jT_i\;\;(j\ge i+2) \\
 T_iT_{i+1}T_i=T_{i+1}T_iT_{i+1}\;\; (1\le i\le n-2).
\end{gather*}
We call $e=\min\{ k\in\N \mid 1+q+\cdots+q^{k-1}=0\;\text{in $K$.}\}$ the \textbf{quantum chracteristic}. 
Let $\mathfrak{g}(A^{(1)}_{e-1})$ be the affine Kac-Moody Lie algebra of type $A^{(1)}_{e-1}$, $\{ \Lambda_i \mid i\in \Z/e\Z\}$ the fundamental weights. 
Then, block algebras of $H^A_n(q)$ $(n=0,1,2,\dots)$ are labeled by the weights of the integrable module $V(\Lambda_0)$. 

The \textbf{Hecke algebra of type $B$} is the $K$-algebra $H_n(q,Q)$, where $1\ne q\in K^\times$ and $Q\in K^\times$,  
defined by generators $T_0, T_1,\dots,T_{n-1}$ and relations
\begin{gather*}
(T_0-Q)(T_0+1)=0,\;\; (T_i-q)(T_i+1)=0\;\;(1\le i\le n-1)\\
(T_0T_1)^2=(T_1T_0)^2,\quad T_iT_j=T_jT_i\;\;(j\ge i+2) \\
 T_iT_{i+1}T_i=T_{i+1}T_iT_{i+1}\;\; (1\le i\le n-2).
\end{gather*}
If $-Q\not\in q^{\Z}$, block algebras are Morita equivalent to tensor product algebras of two block algebras of type $A$ by \cite{DJ92}. 
If $-Q=q^s$, for some $0\le s\le e-1$, 
block algebras are labeled by the weights of the integrable module $V(\Lambda_0+\Lambda_s)$ by \cite{LM07}. 

In type $A$ and type $B$, the affine Weyl group, which is the affine symmetric group generated by Coxeter generators $\{ s_i \mid i\in \Z/e\Z\}$, 
acts on the weights of $V(\Lambda_0)$ and $V(\Lambda_0+\Lambda_s)$. Then, 
block algebras in the same affine Weyl group orbit are mutually derived equivalent by \cite{CR08}. 

The \textbf{Hecke algebra of type $D$} is the $K$-algebra $H^D_n(q)$, where $1\ne q\in K^\times$,  
defined by generators $T_0, T_1,\dots,T_{n-1}$ and relations
\begin{gather*}
(T_i-q)(T_i+1)=0\;\;(0\le i\le n-1),\quad T_0T_i=T_iT_0\;\;(i\ne 2)\\
T_0T_2T_0=T_2T_0T_2, \quad T_iT_j=T_jT_i\;\;(j\ge i+2\ge3) \\
 T_iT_{i+1}T_i=T_{i+1}T_iT_{i+1}\;\; (1\le i\le n-2).
\end{gather*}

Modules are always assumed to be finite dimensional right modules. We call block algebras which are of tame representation type and not of finite representation type simply 
tame block algebras.

\section{Silting theory}

We assume that the reader is familiar with the various theories for Hecke algebras arising from the categorification of integrable modules over the affine Kac-Moody Lie algebra of type 
$A^{(1)}_{e-1}$. However, since experts in the modular representation theory of Hecke algebras are not familiar with new development of the silting theory, 
we briefly review the theory in this section. 

\subsection{}
We start with the definition of silting object and basic properties.

\begin{defn}
An object $X$ of a triangulated category $\mathcal{T}$ is a \textbf{silting object} if 
\begin{itemize}
\item[(i)]
$\Hom_{\mathcal{T}}(X,X[i])=0$, for all $i>0$.
\item[(ii)]
If an additive full subcategory $\mathcal{C}$ of $\mathcal{T}$ satisfies the conditions
\begin{itemize}
\item[(a)]
$\mathcal{C}$ is closed under isomorphism, shift, taking mapping cone, and
\item[(b)]
all the objects of ${\rm add}(X)$ are objects of $\mathcal{C}$.
\end{itemize} 
then we must have $\mathcal{C}=\mathcal{T}$.
\end{itemize}
Furthermore, if indecomposable direct summands of $X$ are pairwise non-isomorphic, then $X$ is called a \textbf{basic silting object}.
If the condition {\rm(i)} is replaced with 
\begin{itemize}
\item[(i)']
$\Hom_{\mathcal{T}}(X,X[i])=0$, for all $i\ne 0$.
\end{itemize}
then $X$ is called a \textbf{tilting object}.
\end{defn}

If a triangulated category $\mathcal{T}$ admits a silting object, then, as the authors of \cite{AI12} pointed out in Remark 2.9 of  their paper, 
the isomorpphism classes of $\mathcal{T}$ form a set by \cite[Prop.2.17]{AI12}. Hence, set theoretical issues do not arise, and we denote the set 
of isomorphism classes of basic silting objects by $Silt(\mathcal{T})$. 
For a finite dimensional algebra $A$, we denote $Silt(K^b(proj(A)))$ by $Silt(A)$. We call silting (resp. tilting) objects 
silting (resp. tilting) complexes when $\mathcal{T}=K^b(proj(A))$. 

The following lemma characterizes tilting complexes among silting complexes. 

\begin{lem}[{\cite[Thm.A.4]{A13}}]
Let $A$ be a finite dimensional selfinjective algebra.  Then, a silting complex $T$ is a tilting complex if and only if $\nu(T)\simeq T$, where $\nu$ is the Nakayama functor.
\end{lem}

\begin{cor}\label{silting equals tilting}
Let $A$ be a finite dimensional symmetric algebra. Then, any silting complex is a tilting complex. 
\end{cor}

\noindent
As we work with finite dimensional symmetric algebras only, all the silting complexes we will consider 
are tilting complexes. The next theorem is well-known.

\begin{thm}[{\cite{R89(a)}\cite{R91}}]\label{Rickard Morita theorem}
Let $A$ and $B$ be finite dimensional selfinjective algebras. Then, they are derived equivalent if and only if there exists a tilting complex $T$ such that 
$B\simeq \End_{K^b(proj(A))}(T)$. Furthermore, there exists a complex of bimodules $X$ in $D^b(B\text{-}mod\text{-}A)$, which is called a two-sided tilting complex, 
such that the derived tensor product with $X$ over $B$ gives the equivalence $D^b(mod(B))\simeq D^b(mod(A))$ which sends the stalk complex $B$ to the tilting complex $T$.
\end{thm}

\subsection{}
Silting objects in a triangulated category are related to each other by silting mutation. 

\begin{defn}
Let $\mathcal{C}$ be an additive category, $X$ and $M$ objects of $\mathcal{C}$. We say that a morphism $f: X\to Y$ is the \textbf{left ${\rm add}(M)$-approximation} of $X$ if 
$Y\in {\rm add}(M)$ and $\Hom(f,U): \Hom_{\mathcal{C}}(Y, U)\to \Hom_{\mathcal{C}}(X, U)$ is surjective,
for all $U\in {\rm add}(M)$. 

If $f$ is further left minimal, that is, if $g\in\Hom_{\mathcal{C}}(Y,Y)$ that satisfies $g\circ f=f$ 
is always an automorphism, we say that $f: X\to Y$ is the \textbf{minimal left ${\rm add}(M)$-approximation} of $X$. 
\end{defn}

\begin{defn}
Let $A$ be a  finite dimensional algebra, and let $T$ be a silting complex. 
We choose an indecomposable direct summand $X$, and write $T=X\oplus M$. We denote by $X\to Y$
the minimal left ${\rm add}(M)$-approximation of $X$, and extend it to a triangle
$X\to Y\to Z \to X[1]$. Then, define $\mu_X(T)=Z\oplus M$ and call $\mu_X(T)$ the \textbf{irreducible left silting mutation} of $T$. 
\end{defn}

\begin{rmk}
For a silting complex $T$, $\mu_X(T)$ is a silting complex by \cite[Thm.2.31]{AI12}. 
For a tilting complex $T$, $\mu_X(T)$ is not necessarily a tilting complex, but if we choose the indecomposable direct summand $X$ to be such that $\nu(X)\simeq X$, 
then $\mu_X(T)$ is a tilting complex by \cite[Lem.5.2]{CKL15}.
\end{rmk}

\begin{rmk}
Replacing the minimal left ${\rm add}(M)$-approximation by the minimal right ${\rm add}(M)$-approximation, we define the \textbf{irreducible right silting mutation}.
\end{rmk}

\begin{thm}[{\cite[Thm.2.11]{AI12}}]
For $T_1, T_2\in Silt(\mathcal{T})$, we write $T_1\ge T_2$ if 
\[
\Hom_{\mathcal{T}}(T_1, T_2[i])=0,\quad \text{for all $i>0$.}
\]
Then, $Silt(\mathcal{T})$ is a partially ordered set. 
\end{thm}

\begin{thm}[{\cite[Thm.2.35, Prop.2.36]{AI12}}]\label{properties of silting order}
Let $A$ be a finite dimensional algebra, and let $T_1$ and $T_2$ be objects of $Silt(A)$. Then, we have the following.
\begin{itemize}
\item[(1)]
If $T_1>T_2$, then there exists an irreducible left silting mutation $T=\mu_X(T_1)$, for an indecomposable direct summand of $T_1$, such that $T_1>T\ge T_2$ holds.
\item[(2)]
The following are equivalent.
\begin{itemize}
\item[(a)]
$T_2$ is an irreducible left silting mutation of $T_1$.
\item[(b)]
$T_1$ is an irreducible right silting mutation of $T_2$.
\item[(c)]
$T_1>T_2$ and there is no silting object $T$ satisfying $T_1>T>T_2$. 
\end{itemize}
\end{itemize}
\end{thm}

\begin{defn}
Let $\mathcal{T}$ be a triangulated category which admits a silting object. We say that $\mathcal{T}$ is \textbf{silting discrete} if, for any silting objects $T_1$ and $T_2$ 
which satisfies $T_1\ge T_2$, there exists only finitely many objects $T$ of $Silt(\mathcal{T})$ that satisfy $T_1\ge T\ge T_2$.
\end{defn}

\noindent
The following proposition is easy to prove. 

\begin{prop}[{\cite[Prop.3.8]{A13}}]\label{one silting object suffices}
A triangulated category $\mathcal{T}$ is silting discrete if and only if there exists a basic silting object $A$ of $\mathcal{T}$ such that 
there are only finitely many objects $T$ of $Silt(\mathcal{T})$ that satisfy $A\ge T\ge A[\ell]$, for any $\ell>0$. 
\end{prop}

\begin{cor}
Let $A$ be a finite dimensional algebra, and view $A$ as a complex concentrated in degree zero. If there are only finitely many silting objects $T\in Silt(A)$ that satisfy 
$A\ge T\ge A[\ell]$, for any $\ell>0$, then $K^b(proj(A))$ is silting discrete.
\end{cor}

\noindent
The meaning of finiteness condition in the definition of silting discreteness is the following. 

\begin{thm}[{\cite[Thm.3.5]{A13}}]\label{iterative left mutation}
Let $A$ be a finite dimensional algebra, $T_1$ and $T_2$ objects of $Silt(A)$ which satisfy $T_1\ge T_2$. If the number of objects $T\in Silt(A)$ that satisfy 
$T_1\ge T\ge T_2$ is finite, then $T_2$ is obtained by iterated irreducible left silting mutation from $T_1$.
\end{thm}

\begin{thm}\label{silting connected}
Let $A$ be a finite dimensional algebra and suppose that $K^b(proj(A))$ is silting discrete. Then, any silting complex is obtained by iterated irreducible left silting mutation
from a shift of the stalk complex $A$. 
\end{thm}
\begin{proof}
For any objects $X, Y\in K^b(proj(A))$, $\Hom_{K^b(proj(A))}(X,Y[i])=0$, for $i>\!>0$. Thus, we fix a sufficiently large $\ell$, and
$\Hom_{K^b(proj(A))}(A[-\ell],X[i])=0$, for all $i>0$. That is, $A[-\ell]\ge X$ holds. Then, the result follows by Theorem \ref{iterative left mutation}.
\end{proof}

\subsection{}
We return to symmetric algebras. The following is an important application of the silting theory. The argument in the proof is taken from \cite[Thm.5,1]{AM17}. 
For symmetric algebras, we say \textbf{tilting mutation} instead of silting mutation.

\begin{thm}\label{unique Morita class}
Let $A_1,\dots,A_s$ be derived equivalent finite dimensional symmetric algebras, and we identify $\mathcal{T}=K^b(proj(A_i))$, for $1\le i\le s$.  Suppose the following. 
\begin{itemize}
\item[(a)]
The triangulated category $\mathcal{T}$ is tilting discrete.
\item[(b)]
For any $1\le i\le s$ and an indecomposable projective $A_i$-module $X$, we have an isomorphism of algebras 
$\End_{\mathcal{T}}(\mu_X(A_i))\simeq A_j$, for some $1\le j\le s$.
\end{itemize}
Then, any finite dimensional algebra $B$ having derived equivalence $K^b(proj(B))\simeq\mathcal{T}$ is Morita equivalent to $A_i$, for some $1\le i\le s$, 
that is, there is a category equivalence $mod(A_i)\simeq mod(B)$, for some $1\le i\le s$.
\end{thm}
\begin{proof}
By Theorem \ref{Rickard Morita theorem}, there is a tilting complex $T\in K^b(proj(A_1))$ such that $B=\End_{\mathcal{T}}(T)$. The condition {\rm(a)} implies that 
$T$ is obtained by iterated irreducible left silting mutation from 
the stalk complex $A_1$, by Theorem \ref{silting connected}, and we write
\[
T\simeq \mu_{X_\ell}\circ\cdots\circ\mu_{X_1}(A_1).
\]
Since $A_1$ is a symmetric algebra, silting complexes are tilting complexes by Corollary \ref{silting equals tilting}, so that 
$T_i=\mu_{X_i}\circ\cdots\circ\mu_{X_1}(A_1)$, for $1\le i\le \ell$, are tilting complexes. 
We show that, for $1\le i\le \ell$, we have $\End_{\mathcal{T}}(T_i)\simeq A_j$, for some $1\le j\le s$. The base $i=1$ is the assumption {\rm (b)}. Suppose that
$\End_{\mathcal{T}}(T_{i-1})\simeq A_k$, for some $1\le k\le s$, holds. Then, Theorem \ref{Rickard Morita theorem} 
implies that there is an auto-equivalence $F:\mathcal{T}\simeq \mathcal{T}$ such that $F(T_{i-1})=A_k$. Hence, we have isomorphisms of finite dimensional algebras
\[
\End_{\mathcal{T}}(T_i)=\End_{\mathcal{T}}(\mu_{X_i}(T_{i-1}))\simeq \End_{\mathcal{T}}(\mu_{F(X_i)}(A_k))
\]
and $\End_{\mathcal{T}}(T_i)\simeq A_j$, for some $1\le j\le s$, by the assumption {\rm (b)} again. 
\end{proof}

As we stated in the introduction, we only need to handle Brauer graph algebras. We define Brauer graph algebras as follows. See \cite{ES17}, for example.

\begin{defn}
A \textbf{Brauer graph} is an undirected graph, which allows loops and multiple edges, such that each vertex $v$ is associated with the multiplicity $m(v)\in\N$, and a cyclic ordering of the edges which have $v$ as an endpoint. Then, the \textbf{Brauer graph algebra} associated with a Brauer graph is defined as follows. 
\begin{itemize}
\item[(a)]
For each vertex $v$, let $\alpha_{v,1}, \cdots, \alpha_{v,c_v}$ be the directed arcs which connect each of the edges in the cyclic ordering around $v$
to the edge which is immediately after the edge in the cyclic ordering. Then, 
\[
\{ \alpha_{v,i} \mid  \text{$v$ is a vertex,}\; 1\le i\le c_v\}
\]
generates the Brauer graph algebra. We call $\alpha_{v,1}, \cdots, \alpha_{v,c_v}$ a \textbf{cycle}. If the cycle starts and ends in $E$, we denote the product 
$\alpha_{v,1}\cdots\alpha_{v,c_v}$ by $C_{E,v}$.
\item[(b)] 

\begin{itemize}
\item[(i)]
If $\alpha_{u,i}\alpha_{v,j}$ is not contained in any cycle, then $\alpha_{u,i}\alpha_{v,j}=0$. 
\item[(ii)]
For the endpoints $u$ and $v$ of an edge $E$, $C_{E,u}^{m(u)}=C_{E,v}^{m(v)}$.
\end{itemize}
\end{itemize}
Note that $\alpha_{v,1}\cdots\alpha_{v,c_v}\alpha_{v,1}=0$ follows from the defining relations. We call vertices of multiplicity strictly greater than one 
\textbf{exceptional vertices}.
\end{defn}

\noindent
Next theorem gives a combinatorial criterion for tilting discreteness of Brauer graph algebras.

\begin{thm}[{\cite[Thm.6.7]{AAC15}}]\label{Aihara thm}
A Brauer graph algebra is tilting discrete if and only if the Brauer graph contains at most one cycle of odd length and no cycle of even length.
\end{thm}

\section{Derived equivalence classes of tame block algebras}

In \cite{Ar17}, the author has determined the affine Weyl group orbit representatives of tame block algebras of Hecke algebras of type $A$ and $B$. 
The representatives are given as follows. The result for type $A$ has been known for a long time. 

\begin{itemize}
\item[(1)]
For Hecke algebras of type $A$, Brauer graph algebras whose Brauer graph are one of the following. (Both are in the same affine Weyl group orbit.)
\[
\begin{xy}
(0,0) *[o]+[Fo]{2}="A", (10,0) *[o]+[Fo]{\hphantom{2}}="B", (20,0) *[o]+[Fo]{2}="C",
\ar@{-} "A";"B"
\ar@{-} "B";"C"
\end{xy}\quad\quad
\begin{xy}
(0,0) *[o]+[Fo]{2}="A", (10,0) *[o]+[Fo]{2}="B", (20,0) *[o]+[Fo]{\hphantom{2}}="C",
\ar@{-} "A";"B"
\ar@{-} "B";"C"
\end{xy}
\]
They occur only when the quantum characteristic is $e=2$.
\item[(2)]
For Hecke algebras of type $B$ with two parameters, either 
\begin{itemize}
\item[(a)]
the Brauer graph algebras in (1), or the symmetric Kronecker algebra $K[X,Y]/(X^2, Y^2)$,  
if the quantum characteristic $e=2$, or
\item[(b)]
the Brauer graph algebra whose Brauer graph is 
\[
\begin{xy}
(0,0) *[o]+[Fo]{2}="A", (10,0) *[o]+[Fo]{2}="B", (20,0) *[o]+[Fo]{2}="C",
\ar@{-} "A";"B"
\ar@{-} "B";"C"
\end{xy}
\]
if the quantum characteristic is $e\ge4$ and $Q=-1$.
\end{itemize}
\end{itemize}

Our aim is to prove that they exhaust Morita equivalence classes of tame block algebras in type $A$ and type $B$. Note that they are tilting discrete by 
Theorem \ref{Aihara thm}. Thus, we compute the endomorphism algebras of irreducible left tilting mutation of the above algebras. 
Then, we apply Theorem \ref{unique Morita class} to obtain the desired result. 

\section{Computation of the endomorphism algebras}

\subsection{}
We begin by the symmetric Kronecker algebra. 
We state the following theorem only for the bounded homotopy category of a finite dimensional algebra, but it is proved in more general setting in \cite{AI12}. 

\begin{thm}[{\cite[Thm.2.26]{AI12}}]\label{local algebra case}
Let $A$ be a finite dimensional algebra. 
If $K^b(proj(A))$ has an indecomposable silting complex, then any silting complex is its shift.
\end{thm}

Hence, the assumptions (a) and (b) of Theorem \ref{unique Morita class} hold for the symmetric Kronecker algebra by 
Theorem \ref{Aihara thm} and Theorem \ref{local algebra case}. Thus, we obtain the following.

\begin{lem}
Let $A$ be  the symmetric Kronecker algebra. If a finite dimensional algebra $B$ is derived equivalent to $A$, then $B$ is Morita equivalent to $A$.
\end{lem}

\subsection{}
Secondly, we consider the Brauer graph algebra for the Brauer graph
\[
\begin{xy}
(0,0) *[o]+[Fo]{2}="A", (10,0) *[o]+[Fo]{2}="B", (20,0) *[o]+[Fo]{2}="C",
\ar@{-} "A";"B"
\ar@{-} "B";"C"
\end{xy}
\]
and we denote it by $A(2,2,2)$. As a bounded quiver algebra, the quiver is 
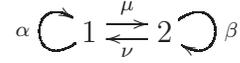
\[
\begin{xy}
(0,0) *[o]+{1}="A", (10,0) *[o]+{2}="B",
\ar @(dl,ul) "A";"A"^{\alpha}
\ar @(ur,dr) "B";"B"^{\beta}
\ar @<1mm> "A";"B"^\mu
\ar @<1mm> "B";"A"^\nu
\end{xy}
\]
and the relations are
\[
\alpha\mu=\mu\beta=\beta\nu=\nu\alpha=0,\;\; \alpha^2=(\mu\nu)^2, \;\; \beta^2=(\nu\mu)^2.
\]
We denote $A(2,2,2)$ by $A$. The indecomposable projective $A$-modules are
\begin{align*}
P_1&=\Span\{e_1, \alpha, \mu, \mu\nu, \mu\nu\mu, (\mu\nu)^2\},\\
P_2&=\Span\{e_2, \beta, \nu, \nu\mu, \nu\mu\nu, (\nu\mu)^2\}.
\end{align*}
The heart $\Rad(P_1)/\Soc(P_1)$ is the direct sum of $\Span\{\alpha\}$ and the uniserial module $\Span\{\mu,\mu\nu, \mu\nu\mu\}$. 
Similarly, $\Rad(P_2)/\Soc(P_2)$ is the direct sum of $\Span\{\beta\}$ and the uniserial module $\Span\{\nu,\nu\mu, \nu\mu\nu\}$. 
$\Hom_A(P_1,P_2)$ consists of linear combinations of left multiplication by $\nu$ and $\nu\mu\nu$. We denote it by
\[
\Hom_A(P_1,P_2)=\Span\{\nu, \nu\mu\nu\}.
\]
Then, $\Hom_A(P_2,P_1)=\Span\{\mu, \mu\nu\mu\}$, and 
\[
\End_A(P_1)=\Span\{e_1, \mu\nu, \alpha, \alpha^2 \}, \;\; \End_A(P_2)=\Span\{e_2, \nu\mu, \beta, \beta^2\}.
\]

We may show that the derived equivalence class of $A(2,2,2)$ coincides with the Morita equivalence class of $A(2,2,2)$ as follows. 

\begin{prop}\label{Endoalgebra computation}
Any finite dimensional algebra which is derived equivalent to $A(2,2,2)$ is Morita equivalent to $A(2,2,2)$.
\end{prop}
\begin{proof}
Let $A=A(2,2,2)$. We compute the endomorphism algebra $\End_{K^b(A)}(\mu_{P_1}(A))$. The computation of $\End_{K^b(A)}(\mu_{P_2}(A))$ is obtained by swapping the role of $P_1$ and $P_2$. First of all, it is easy to see that the minimal left ${\rm add}(P_2)$-approximation is the left multiplication by $\nu$, and we denote it by $\nu: P_1\to P_2$. Thus, 
the irreducible left tilting mutation $\mu_{P_1}(A)$ is the complex
\[
\begin{array}{ccccccccc}
\cdots\rightarrow & 0 & \rightarrow & P_1 & \rightarrow & P_2\oplus P_2 & \rightarrow & 0 & \rightarrow\cdots
\end{array}
\]
where the differential  $d: P_1 \rightarrow P_2\oplus P_2$ from degree $-1$ to degree $0$ is given by
\[
d=\begin{pmatrix}
\nu \\ 0
\end{pmatrix}.
\]
We consider the space of endomorphisms $\{ f_i\}_{i\in\Z}$ of complexes:
\[
\begin{array}{ccccccccc}
\cdots\rightarrow & 0 & \rightarrow & P_1 & \rightarrow & P_2\oplus P_2 & \rightarrow & 0 & \rightarrow\cdots\\
          & \downarrow &      & \downarrow &                  & \downarrow &  & \downarrow & \\
\cdots\rightarrow & 0 & \rightarrow & P_1 & \rightarrow & P_2\oplus P_2 & \rightarrow & 0 & \rightarrow\cdots
\end{array}
\]
Since $f_i=0$, for $i\ne-1,0$, we write elements of $\End_{C^b(proj(A))}(\mu_{P_1}(A))$ by
\[
f=\begin{pmatrix} f_{-1} & 0 \\ 0 & f_0 \end{pmatrix}.
\]
Then $\End_{C^b(proj(A))}(\mu_{P_1}(A))$ is the matrix algebra consisting of the elements
\[
\begin{pmatrix}
a_1e_1+a_2\mu\nu+a_5\alpha+a_6\alpha^2 & 0 & 0 \\
0 & a_1e_2+a_2\nu\mu+a_3\beta+a_4\beta^2 & b_1e_2+b_2\nu\mu+b_3\beta+b_4\beta^2 \\
0 & c_3\beta+c_4\beta^2 & d_1e_2+d_2\nu\mu+d_3\beta+d_4\beta^2
\end{pmatrix}
\]

\noindent
where $a_i, b_i, c_i, d_i$ are coefficients. The null-homotopic endomorphisms form 
its ideal consisting of the elements
\[
\begin{pmatrix}
p\mu\nu & 0 & 0 \\
0 & p\nu\mu+q\beta^2 & r\nu\mu+s\beta^2 \\ 
0 & 0 & 0 
\end{pmatrix}
\]
where $p,q,r,s$ are coefficients, and the factor algebra is $\End_{K^b(proj(A))}(\mu_{P_1}(A))$. Now we observe that 
$\End_{K^b(proj(A))}(\mu_{P_1}(A))$ is generated by
\begin{align*}
e_1'=\begin{pmatrix} e_1 & 0 & 0 \\ 0 & e_2 & 0 \\ 0 & 0 & 0 \end{pmatrix}, \quad& e_2'=\begin{pmatrix} 0 & 0 & 0 \\ 0 & 0 & 0 \\ 0 & 0 & e_2 \end{pmatrix} \\[5pt]
\alpha'=\begin{pmatrix} \alpha & 0 & 0 \\ 0 & e_2 & 0 \\ 0 & 0 & 0 \end{pmatrix}, \quad& \mu'=\begin{pmatrix} 0 & 0 & 0 \\ 0 & 0 & e_2 \\ 0 & 0 & 0 \end{pmatrix} \\[5pt]
\nu'=\begin{pmatrix} 0 & 0 & 0 \\ 0 & 0 & 0 \\ 0 & \beta & 0 \end{pmatrix}, \quad& \beta'=\begin{pmatrix} 0 & 0 & 0 \\ 0 & 0 & 0 \\ 0 & 0 & \nu\mu \end{pmatrix}
\end{align*}
and they satisfy the same relations as the generators $e_1, e_2, \alpha, \mu, \nu, \beta$ of $A(2,2,2)$ but ${\alpha'}^2=-(\mu'\nu')^2$. 
Thus, we conclude that $\End_{K^b(proj(A))}(\mu_{P_1}(A))\simeq A(2,2,2)$.
\end{proof}

\subsection{}
We consider 
$\begin{xy}
(0,0) *[o]+[Fo]{2}="A", (10,0) *[o]+[Fo]{2}="B", (20,0) *[o]+[Fo]{\hphantom{2}}="C",
\ar@{-} "A";"B"
\ar@{-} "B";"C"
\end{xy}$ and 
$\begin{xy}
(0,0) *[o]+[Fo]{2}="A", (10,0) *[o]+[Fo]{\hphantom{2}}="B", (20,0) *[o]+[Fo]{2}="C",
\ar@{-} "A";"B"
\ar@{-} "B";"C"
\end{xy}$. We denote the corresponding Brauer graph algebras by $A(2,2,1)$ and $A(2,1,2)$, respectively. 

\begin{prop}
We have the following.
\begin{itemize}
\item[(1)]
Let $A=A(2,2,1)$ and $P$ an indecomposable projective $A$-module. Then, $\End_{K^b(proj(A))}(\mu_P(A))$ is isomorphic to $A(2,2,1)$ or $A(2,1,2)$.
\item[(2)]
Let $A=A(2,1,2)$ and $P$ an indecomposable projective $A$-module. Then, $\End_{K^b(proj(A))}(\mu_P(A))$ is isomorphic to $A(2,2,1)$.
\item[(3)]
Finite dimensional algebras in the derived equivalence class of $A(2,1,2)$ are Morita equivalent to either $A(2,2,1)$ or $A(2,1,2)$.
\end{itemize}
\end{prop}
\begin{proof}
The computation of tilting mutation for the algebras $A(2,2,1)$ and $A(2,1,2)$ in (1) and (2) are similar to the proof of Proposition \ref{Endoalgebra computation}. 
Thus, we only give the result of the computation. 

{\rm(1)} Let $A=A(2,2,1)$. Then, $A$ is the bounded quiver algebra whose quiver is 
\[
\begin{xy}
(0,0) *[o]+{1}="A", (10,0) *[o]+{2}="B",
\ar @(dl,ul) "A";"A"^{\alpha}
\ar @<1mm> "A";"B"^\mu
\ar @<1mm> "B";"A"^\nu
\end{xy}
\]
and the relations are $\alpha\mu=\mu\nu\mu\nu\mu=\nu\mu\nu\mu\nu=\nu\alpha=0$, $\alpha^2=(\mu\nu)^2$. 
We start with $\mu_{P_1}(A)$. Then, $\End_{C^b(proj(A))}(\mu_{P_1}(A))$ is the matrix algebra consisting of 
\[
\begin{pmatrix}
a_1e_1+a_2\mu\nu+a_4\alpha+a_5\alpha^2 & 0 & 0 \\
0 & a_1e_2+a_2\nu\mu+a_3(\nu\mu)^2 & b_1e_2+b_2\nu\mu+b_3(\nu\mu)^2 \\
0 & c_1(\nu\mu)^2 & d_1e_2+d_2\nu\mu+d_3(\nu\mu)^2
\end{pmatrix}
\]
and the two-sided ideal of null-homotopic elements consists of
\[
\begin{pmatrix}
p\mu\nu & 0 & 0 \\
0& p\nu\mu+q(\nu\mu)^2 & r\nu\mu + s(\nu\mu)^2 \\
0& 0 & 0 
\end{pmatrix}.
\]
We define basis elements of $\End_{K^b(proj(A))}(\mu_{P_1}(A))$ in the same way as in the proof of Proposition \ref{Endoalgebra computation}, except for 
$\nu'$. For $\nu'$, we replace $\beta$ with $(\nu\mu)^2$. By modifying the sign of ${\alpha'}^2=-\mu'\nu'$, we conclude that 
$\End_{K^b(proj(A))}(\mu_{P_1}(A))\simeq A(2,1,2)$.

For $\End_{K^b(proj(A))}(\mu_{P_2}(A))$, we do the similar computation and define
\begin{gather*}
e_1'=\begin{pmatrix} e_2 & 0 & 0 \\ 0 & e_1 & 0 \\ 0 & 0 & 0 \end{pmatrix}, \quad e_2'=\begin{pmatrix} 0 & 0 & 0 \\ 0 & 0 & 0 \\ 0 & 0 & e_1 \end{pmatrix} \\[5pt]
\mu'=\begin{pmatrix} 0 & 0 & 0 \\ 0 & 0 & e_1 \\ 0 & 0 & 0 \end{pmatrix} \quad
\nu'=\begin{pmatrix} 0 & 0 & 0 \\ 0 & 0 & 0 \\ 0 & \alpha & 0 \end{pmatrix}, \quad \beta'=\begin{pmatrix} 0 & 0 & 0 \\ 0 & 0 & 0 \\ 0 & 0 & \mu\nu \end{pmatrix}.
\end{gather*}
Then, they define a bounded quiver algebra for the quiver
\[
\begin{xy}
(0,0) *[o]+{2}="A", (10,0) *[o]+{1}="B",
\ar @(dl,ul) "A";"A"^{\beta'}
\ar @<1mm> "A";"B"^{\nu'}
\ar @<1mm> "B";"A"^{\mu'}
\end{xy}
\]
and it follows that $\End_{K^b(proj(A))}(\mu_{P_2}(A))\simeq A(2,2,1)$.

{\rm (2)} Let $A=A(2,1,2)$. Then, $A$ is the bounded quiver algebra whose quiver is
\[
\begin{xy}
(0,0) *[o]+{1}="A", (10,0) *[o]+{2}="B",
\ar @(dl,ul) "A";"A"^{\alpha}
\ar @(ur,dr) "B";"B"^{\beta}
\ar @<1mm> "A";"B"^\mu
\ar @<1mm> "B";"A"^\nu
\end{xy}
\]
and whose relations are $\alpha\mu=\mu\beta=\beta\nu=\nu\alpha=0,\;\; \alpha^2=\mu\nu, \;\; \beta^2=\nu\mu$. Since the computation is symmetric, 
it suffices to consider $\mu_{P_1}(A)$. Then, the similar computation above shows that $\End_{K^b(proj(A))}(\mu_{P_1}(A))\simeq A(2,2,1)$.

{\rm(3)} Since Theorem \ref{Aihara thm} implies that these algebras are tilting discrete, 
(3) follows because the algebras satisfy the assumptions (a) and (b) of Theorem \ref{unique Morita class}.
\end{proof}

\section{Tame block algebras of Hecke algebras of type $D$}

In this section, we consider block algebras of Hecke algebras of  type $D$. Thus, we consider the Hecke algebra $H_n(q,1)$ of type $B$ for the parameter $Q=1$. 
The algebra $H_n(q,1)$ is generated by $T_0,\dots,T_{n-1}$ and the quadratic equation for $T_0$ is $T_0^2=1$. 
Define an algebra automorphism $\tau$ of $H_n(q,1)$ by $\tau(T_1)=T_0T_1T_0$ and $\tau(T_i)=T_i$, for $i\ne1$. 
Define another algebra automorphism $\sigma$ of $H_n(q,1)$ by $\sigma(T_0)=-T_0$ and $\sigma(T_i)=T_i$, for $i\ne0$. 
Then, $\sigma\tau=\tau\sigma$ and the Hecke algebra of type $D$ is the fixed point subalgebra $H_n(q,1)^\sigma$. 

Recall from the author's work \cite{Ar01(b)} that irreducible $H_n(q,1)$-modules are labeled by Kleshchev bipartitions when 
$e\ge2$ is even. They are nodes of the Misra-Miwa realization of the Kashiwara crystal $B(\Lambda)$, for $\Lambda=\Lambda_0+\Lambda_{e/2}$. 
The signature rule to compute the Kashiwara operators, for a given bipartition, is as follows. 
\begin{itemize}
\item[(a)]
Read removable and indent $i$-nodes $(i\in\Z/2\Z)$ from the top row of the first component of the bipartition to the bottom row of the second 
component of the bipartition. 
\item[(b)]
Delete consecutive occurrence of a removable $i$-node and an indent $i$-node in this order as many times as possible from the sequence, 
and change the status of the rightmost indent $i$-node to removable $i$-node.
\end{itemize}
We denote the block algebra of $H_n(q,1)$ labeled by a weight $\Lambda-\beta$ of the integrable highest weight module $V(\Lambda)$ by $R^\Lambda(\beta)$. 

For a Kleshchev bipartition $\lambda\vdash n$, we denote the irreducible $H_n(q,1)$-module by $D^\lambda$. 
Then, we denote by $(D^\lambda)^\sigma$ the irreducible $H_n(q,1)$-module obtained from $D^\lambda$ by twisting the module structure by $\sigma$, and 
define $h(\lambda)$ by $(D^\lambda)^\sigma=D^{h(\lambda)}$. 
The next theorem is obtained by a version of Clifford theory.

\begin{thm}\label{restriction}
Recall that the base field is algebraically closed of characterisitc not equal to two. 
\begin{itemize}
\item[(1)]
If $h(\lambda)\ne\lambda$ then $D^\lambda$ remains irreducible as an $H_n(q,1)^\sigma$-module. Further, 
$D^\lambda$ and $D^{h(\lambda)}$ are equivalent as $H_n(q,1)^\sigma$-modules. 
\item[(2)]
If $h(\lambda)\ne\lambda$ then $D^\lambda$ is the direct sum of pairwise inequivalent irreducible $H_n(q,1)^\sigma$-modules. Further, 
the twist by $\tau$ swaps the two irreducible $H_n(q,1)^\sigma$-modules. 
\end{itemize}
\end{thm}

\noindent
The following result of Hu enables us to compute $h(\lambda)$ explicitly.

\begin{thm}[{\cite[Thm.1.5]{Hu03}}]
Assume that $e\ge2$ is even and $\Lambda=\Lambda_0+\Lambda_{e/2}$. 
If $\lambda=\tilde{f}_{i_1}\cdots \tilde{f}_{i_n}\emptyset\in B(\Lambda)$, then $h(\lambda)=\tilde{f}_{i_1+e/2}\cdots \tilde{f}_{i_n+e/2}\emptyset$. 
\end{thm}

If $e$ is odd, $-Q\not\in q^{\Z}$ implies that tame block algebras are Morita equivalent to the symmetric Kronecker algebra but it occurs only when $e=2$, contradicting 
our assumption that $e$ is odd. Hence $e$ must be even, and $-Q=q^{e/2}$. Then, we must have $e=2$ again by \cite[Thm.A]{Ar17}. 
In this situation, the following proposition holds. 

\begin{prop}\label{type D indec}
Let $A$ be a tame block algebra of type $D$, and $B$ the block algebra of type $B$ that covers $A$. 
Then, irreducible $B$-modules remain irreducible if we view them as $A$-modules.
\end{prop}
\begin{proof}
As we work in the case $e=2$, we may enumerate the affine Weyl group orbits explicitly. Let $\Lambda=\Lambda_0+\Lambda_1$ and 
$\{\alpha_0,\alpha_1\}$ the simple roots, $\delta=\alpha_0+\alpha_1$ the null root. 
Then, we may prove the formulas below by induction on $k\ge0$.
\begin{align*}
\Lambda-(s_0s_1)^{k+1}\Lambda&=(k+1)(2k+3)\alpha_0+(k+1)(2k+1)\alpha_1,\\
\Lambda-(s_1s_0)^{k+1}\Lambda&=(k+1)(2k+1)\alpha_0+(k+1)(2k+3)\alpha_1,\\
\Lambda-(s_0s_1)^ks_0\Lambda&=k(2k+1)\alpha_0+(k+1)(2k+1)\alpha_1,\\
\Lambda-(s_1s_0)^ks_1\Lambda&=(k+1)(2k+1)\alpha_0+k(2k+1)\alpha_1.
\end{align*}
\begin{itemize}
\item[(i)]
Let $\beta=(k+1)(2k+3)\alpha_0+(k+1)(2k+1)\alpha_1+\delta$. Then, $B=R^\Lambda(\beta)$ and the bipartitions
\begin{align*}
\lambda_1&=((2k+1,2k,\dots,1,1,1), (2k+2,2k+1,\dots,3,2,1))\\
&=\tilde{f}_0^{4k+3}\tilde{f}_1^{4k+1}\cdots \tilde{f}_0^{3}\tilde{f}_1^{2}\tilde{f}_0\emptyset
=\tilde{f}_0^{\max}\tilde{f}_1^{\max}\cdots \tilde{f}_0^{\max}\tilde{f}_1^{2}\tilde{f}_0\emptyset,\\[5pt]
\lambda_2&=((2k+1,2k,\dots,1), (2k+2,2k+1,\dots,3,2,1,1,1))\\
&=\tilde{f}_0^{4k+3}\tilde{f}_1^{4k+1}\cdots \tilde{f}_0^{3}\tilde{f}_1\tilde{f}_0\tilde{f}_1\emptyset
=\tilde{f}_0^{\max}\tilde{f}_1^{\max}\cdots \tilde{f}_0^{\max}\tilde{f}_1\tilde{f}_0\tilde{f}_1\emptyset
\end{align*}
label irreducible $B$-modules. By transposing the first component of $\lambda_1$ and the second component of $\lambda_2$, we obtain other two bipartitions that belong 
to this block. The last bipartition that belongs to this block is the bipartition of two $2$-cores
$((2k+3,2k+2,\dots,1), (2k,2k-1,\dots,1))$. Then,
\begin{align*}
h(\lambda_1)&=\tilde{f}_1^{\max}\tilde{f}_0^{\max}\cdots \tilde{f}_1^{\max}(\emptyset,(2,1))\\
&=((2k,2k-1,\dots,1), (2k+3,2k+2,\dots,1))\ne \lambda_1,\\[5pt]
h(\lambda_2)&=\tilde{f}_1^{\max}\tilde{f}_0^{\max}\cdots \tilde{f}_1^{\max}((1), (1,1))\\
&=((2k+2,2k+1,\dots,1), (2k+1,2k,\dots,1,1))\ne \lambda_2.
\end{align*}
Thus, Theorem \ref{restriction} implies that the irreducible $B$-modules remain irreducible as $A$-modules.
\item[(ii)]
Let $\beta=(k+1)(2k+1)\alpha_0+(k+1)(2k+3)\alpha_1+\delta$. Then, Kleshchev bipartitions
\begin{align*}
\lambda_1&=((2k+2,2k+1,\dots,2,1), (2k+1,2k,\dots,2,1,1,1))\\
&=\tilde{f}_1^{4k+3}\tilde{f}_0^{4k+1}\cdots \tilde{f}_1^{3}\tilde{f}_0\tilde{f}_1\tilde{f}_0\emptyset
=\tilde{f}_1^{\max}\tilde{f}_0^{\max}\cdots \tilde{f}_1^{\max}((1),(1,1)),\\[5pt]
\lambda_2&=((2k,2k-1,\dots,2,1), (2k+3,2k+2,\dots,2,1))\\
&=\tilde{f}_1^{4k+3}\tilde{f}_0^{4k+1}\cdots \tilde{f}_1^{3}\tilde{f}_0^{2}\tilde{f}_1\emptyset
=\tilde{f}_1^{\max}\tilde{f}_0^{\max}\cdots \tilde{f}_1^{\max}(\emptyset,(2,1)),
\end{align*}
where we may also write
\[
\lambda_1=\tilde{f}_1\tilde{f}_0\tilde{f}_1^{4k+3}\tilde{f}_0^{4k+1}\cdots \tilde{f}_1^{3}\tilde{f}_0\emptyset
=\tilde{f}_1\tilde{f}_0\tilde{f}_1^{\max}\tilde{f}_0^{\max}\cdots \tilde{f}_1^{\max}\tilde{f}_0^{\max}\emptyset, 
\]
label irreducible $B$-modules and we have $h(\lambda_1)\ne\lambda_1$ and  $h(\lambda_2)\ne\lambda_2$.
\item[(iii)]
Let $\beta=k(2k+1)\alpha_0+(k+1)(2k+1)\alpha_1+\delta$. Then, Kleshchev bipartitions
\begin{align*}
\lambda_1&=((2k,2k-1,\dots,2,1,1,1), (2k+1,2k,\dots,2,1))\\
&=\tilde{f}_1^{4k+1}\tilde{f}_0^{4k-1}\cdots \tilde{f}_0^{3}\tilde{f}_1^{2}\tilde{f}_0\emptyset
=\tilde{f}_1^{\max}\tilde{f}_0^{\max}\cdots \tilde{f}_0^{\max}((1,1),(1)),\\[5pt]
\lambda_2&=((2k,2k-1,\dots,2,1), (2k+1,2k,\dots,2,1,1,1))\\
&=\tilde{f}_1^{4k+1}\tilde{f}_0^{4k-1}\cdots \tilde{f}_0^{3}\tilde{f}_1\tilde{f}_0\tilde{f}_1\emptyset
=\tilde{f}_1^{\max}\tilde{f}_0^{\max}\cdots \tilde{f}_0^{\max}(\emptyset, (1,1,1))
\end{align*}
label irreducible $B$-modules and we have $h(\lambda_1)\ne\lambda_1$ and  $h(\lambda_2)\ne\lambda_2$.
\item[(iv)]
Let $\beta=(k+1)(2k+1)\alpha_0+k(2k+1)\alpha_1+\delta$. Then, Kleshchev bipartitions
\begin{align*}
\lambda_1&=((2k+1,2k,\dots,3,2,1), (2k,2k-1,\dots,1,1,1))\\
&=\tilde{f}_0\tilde{f}_1\tilde{f}_0^{4k+1}\tilde{f}_1^{4k-1}\cdots \tilde{f}_1^{3}\tilde{f}_0\emptyset
=\tilde{f}_0\tilde{f}_1\tilde{f}_0^{\max}\tilde{f}_1^{\max}\cdots \tilde{f}_1^{\max}\tilde{f}_0^{\max}\emptyset\\
&=\tilde{f}_0^{4k+1}\tilde{f}_1^{4k-1}\cdots \tilde{f}_1^{3}\tilde{f}_0\tilde{f}_1\tilde{f}_0\emptyset
=\tilde{f}_0^{\max}\tilde{f}_1^{\max}\cdots \tilde{f}_1^{\max}((1), (1,1)),\\[5pt]
\lambda_2&=((2k-1,2k-2,\dots,1), (2k+2,2k+1,\dots,3,2,1))\\
&=\tilde{f}_0^{4k+1}\tilde{f}_1^{4k-1}\cdots \tilde{f}_1^{3}\tilde{f}_0^2\tilde{f}_1\emptyset
=\tilde{f}_0^{\max}\tilde{f}_1^{\max}\cdots \tilde{f}_1^{\max}(\emptyset, (2,1))
\end{align*}
label irreducible $B$-modules and we have $h(\lambda_1)\ne\lambda_1$ and  $h(\lambda_2)\ne\lambda_2$.
\end{itemize}
Hence, the irreducible $B$-modules do not split in all the cases.
\end{proof}

\noindent
Proposition \ref{type D indec} implies the desired result for tame block algebras of type $D$.

\begin{cor}
Any tame block algebra of type $D$ is Morita equivalent to either $A(2,2,1)$ or $A(2,1,2)$. 
\end{cor}

\section{Decomposition numbers}

The Cartan matrices for $A(2,1,2)$, $A(2,2,1)$ and $A(2,2,2)$ are $C=(\begin{smallmatrix} 3 & 1 \\ 1 & 3 \end{smallmatrix})$,  $(\begin{smallmatrix} 4 & 2 \\ 2 & 3 \end{smallmatrix})$ 
and $(\begin{smallmatrix} 4 & 2 \\ 2 & 4 \end{smallmatrix})$, respectively. 
Then the decomposition matrix $D$ is determined modulo permutation of the rows by the equation $D^TD=C$, for each of the three cases as follows. 
\[
\begin{pmatrix}
1 & 0 \\
1 & 0 \\
1 & 1 \\
0 & 1 \\
0 & 1
\end{pmatrix}
\;\;\text{for $A(2,1,2)$}, \quad
\begin{pmatrix}
1 & 0 \\
1 & 0 \\
1 & 1 \\
1 & 1 \\
0 & 1
\end{pmatrix}
\;\;\text{for $A(2,2,1)$}, \quad
\begin{pmatrix}
1 & 0 \\
1 & 0 \\
1 & 1 \\
1 & 1 \\
0 & 1 \\
0 & 1
\end{pmatrix}
\;\;\text{for $A(2,2,2)$}.
\]
The decomposition matrix given in Theorem \ref{Thm1} is determined by the same method. 

\section{An example from wild block algebras}

The main result of this paper shows that tame block algebras exhaust Morita classes in each of the derived equivalence class. 
We give an example that it is not the case for wild block algebras as we may naturally expect. 
In this section, we use left modules following the standard convention. But it is harmless 
since the opposite algebra of a cyclotomic quiver Hecke algebra is isomorphic to the original algebra. 

Let $q$ be a primitive third root of unity, and we consider the block algebra of the Hecke algebra of type $B$ with $Q=-q$ labeled by the weight 
$\Lambda-\delta$, where $\Lambda=\Lambda_0+\Lambda_1$ and $\delta=\alpha_0+\alpha_1+\alpha_2$. The algebra is a special case of the 
cyclotomic quiver Hecke algebra $R^\Lambda(\delta)$ of the Lie type $A^{(1)}_{\ell=2}$ associated with 
\[
Q_{ij}(u,v)=\begin{cases}
u+v \quad & (i,j)=(0,1), (1,2), (1,0), (2,1),\\
\lambda u+v &(i,j)=(0,2), \\
u+\lambda v &(i,j)=(2,0), \\
1 & \text{otherwise.}
\end{cases}
\]
where $\lambda$ is a nonzero parameter. The graded dimension formula, which is proven in the same way as \cite[Theorem 3.5]{AIP15}, shows that 
an idempotent generator $e(\nu)$ of $R^\Lambda(\delta)$ is nonzero if and only if $\nu$ is one of 
\[
\nu[1]=(0,2,1), \; \nu[2]=(0,1,2), \; \nu[3]=(1,0,2), \; \nu[4]=(1,2,0), 
\]
and if we denote $e(\nu[i])$ by $e_i$ then 
\begin{align*}
\dim_q e_1R^\Lambda(\delta)e_1&=1+q^2+q^4,\;\; \dim_q e_2R^\Lambda(\delta)e_1=q+q^3, \\
\dim_q e_3R^\Lambda(\delta)e_1&=q^2,\;\; \dim_q e_4R^\Lambda(\delta)e_1=0, \\[5pt]
\dim_q e_1R^\Lambda(\delta)e_2&=q+q^3,\;\; \dim_q e_2R^\Lambda(\delta)e_2=1+2q^2+q^4, \\
\dim_q e_3R^\Lambda(\delta)e_2&=q+q^3,\;\; \dim_q e_4R^\Lambda(\delta)e_2=q^2,\\[5pt]
\dim_q e_1R^\Lambda(\delta)e_3&=q^2,\;\; \dim_q e_2R^\Lambda(\delta)e_3=q+q^3, \\
\dim_q e_3R^\Lambda(\delta)e_3&=1+2q^2+q^4, \;\; \dim_q e_4R^\Lambda(\delta)e_3=q+q^3, \\[5pt]
\dim_q e_1R^\Lambda(\delta)e_4&=0,\;\; \dim_q e_2R^\Lambda(\delta)e_4=q^2, \\
\dim_q e_3R^\Lambda(\delta)e_4&=q+q^3, \;\; \dim_q e_4R^\Lambda(\delta)e_4=1+q^2+q^4. 
\end{align*}

We consider other generators $x_1,x_2,x_3$ and $\psi_1,\psi_2$. First of all, it is clear that $x_1e_i=0$, for $1\le i\le 4$. 
\begin{itemize}
\item[(1)]
We start with 
\begin{gather*}
\psi_1e_1=e(2,0,1)\psi_1=0, \;\; x_2e_1=(\lambda x_1+x_2)e_1=\psi_1^2e_1=0, \\
x_1\psi_2e_1=\psi_2x_1e_1=0,\;\; x_3\psi_2e_1=\psi_2x_2e_1=0,\\
x_1\psi_1\psi_2e_1=x_1e_4\psi_1\psi_2=0,\;\; x_2\psi_1\psi_2e_1=\psi_1x_1e_2\psi_2=0,\\
x_3\psi_1\psi_2e_1=\psi_1\psi_2x_2e_1=0.
\end{gather*}
It follows that $R^\Lambda(\delta)e_1$ is equal to
\begin{gather*}
K[x_1,x_2,x_3]\Span\{ e_1, \psi_1e_1, \psi_2e_1, \psi_1\psi_2e_1, \psi_2\psi_1e_1, \psi_1\psi_2\psi_1e_1\}\\
=K[x_3]e_1+K[x_2]e_2\psi_2e_1+Ke_3\psi_1\psi_2e_1.
\end{gather*}
Moreover, $x_2^2e_2=x_2\psi_1^2e_2=\psi_1x_1e_4\psi_1=0$ implies $x_2^2\psi_2e_1=x_2^2e_2\psi_2=0$ and 
$x_3^3e_1=x_3^2(x_2+x_3)e_1=x_3^2\psi_2^2e_1=\psi_2x_2^2e_2\psi_2=0$.\footnote{These also follow from the graded dimensions.} Hence,
\[
R^\Lambda(\delta)e_1=\Span\{ e_1, e_2\psi_2e_1, x_3e_1, e_3\psi_1\psi_2e_1, x_2e_2\psi_2e_1, x_3^2e_1\}.
\]
\item[(2)]
Using $\psi_1\psi_2e_2=e(2,0,1)\psi_1\psi_2=0$, $x_2\psi_2e_2=x_2e_1\psi_2=0$ and $x_1e_i=0$,
\begin{align*}
R^\Lambda(\delta)e_2&=K[x_1,x_2,x_3]\Span\{ e_2, \psi_1e_2, \psi_2e_2, \psi_1\psi_2e_2, \psi_2\psi_1e_2, \psi_2\psi_1\psi_2e_2\}\\
&=K[x_2,x_3]e_2+K[x_3]e_3\psi_1e_2+K[x_3]e_1\psi_2e_2+K[x_2,x_3]e_4\psi_2\psi_1e_2.
\end{align*}
Then one can show that $R^\Lambda(\delta)e_2$ is equal to
\[
\Span\{ e_2, e_1\psi_2e_2, e_3\psi_1e_2, x_2e_2, x_3e_2, e_4\psi_2\psi_1e_2, x_3e_1\psi_2e_2, x_3e_3\psi_1e_2, x_2x_3e_2 \}
\]
where $x_2x_3e_2=-x_3^2e_2$ by $x_2x_3e_2+x_3^2e_2=x_3\psi_2^2e_2=\psi_2x_2e_1\psi_2=0$ and $x_2^2e_2=x_2\psi_1^2e_2=\psi_1x_1e_3\psi_1=0$.
\item[(3)]
Nextly, $x_2e_4=\psi_1^2e_4=\psi_1e(2,1,0)\psi_1=0$ implies $x_2\psi_2e_3=x_2e_4\psi_2=0$. We also have $\psi_1\psi_2e_3=e(2,1,0)\psi_1\psi_2=0$.
It follows that 
\begin{align*}
R^\Lambda(\delta)e_3&=K[x_1,x_2,x_3]\Span\{ e_3, \psi_1e_3, \psi_2e_3, \psi_1\psi_2e_3, \psi_2\psi_1e_3, \psi_2\psi_1\psi_2e_3\}\\
&=K[x_2,x_3]e_3+K[x_3]e_2\psi_1e_3+K[x_3]e_4\psi_2e_3+K[x_2,x_3]e_1\psi_2\psi_1e_3.
\end{align*}
Thus, one can show that $R^\Lambda(\delta)e_3$ is equal to
\[
\Span\{ e_3, e_2\psi_1e_3, e_4\psi_2e_3, x_2e_3, x_3e_3, e_1\psi_2\psi_1e_3, x_3e_2\psi_1e_3, x_3e_4\psi_2e_3, x_2x_3e_3 \}
\]
where $x_2x_3e_3=-\lambda^{-1}x_3^2e_3$ by $\lambda x_2x_3e_3+x_3^2e_3=x_3\psi_2^2e_3=\psi_2x_2e_4\psi_2=0$ and $x_2^2e_3=x_2\psi_1^2e_3=\psi_1x_1e_2\psi_1=0$. 
\item[(4)]
Finally, $\psi_1e_4=0$, $x_2e_4=\psi_1^2e_4=0$ and $x_3\psi_2e_4=\psi_2x_2e_4=0$ imply
\begin{align*}
R^\Lambda(\delta)e_4&=K[x_3]e_4+K[x_2]e_3\psi_2e_4+K[x_2,x_3]e_2\psi_1\psi_2e_4 \\
&=\Span\{ e_4, e_3\psi_2e_4, x_3e_4, e_2\psi_1\psi_2e_4, x_2e_3\psi_2e_4, x_3^2e_4 \}.
\end{align*}
\end{itemize}

By the graded dimensions, the radical of $R^\Lambda(\delta)$ is spanned by elements of positive degree, and it follows that $R^\Lambda(\delta)$ is a basic algebra and 
$R^\Lambda e_i$, for $1\le i\le 4$, form a complete set of indecomposable projective $R^\Lambda(\delta)$-modules. 
Recall that the cyclotomic quiver Hecke algebra admits an anti-involution 
which fixes each of the generators. Thus, $R^\Lambda(\delta)$ is isomorphic to its opposite algebra, and it follows that we have the bounded quiver presentation of 
$R^\Lambda(\delta)$ as follows. 

\begin{lem}\label{quiver presentation}
Let $Q$ be the quiver 
\[
\begin{xy}
(0,0) *[o]+{1}="A", (10,0) *[o]+{2}="B", (20,0) *[o]+{3}="C", (30,0) *[o]+{4}="D",
\ar@ <1mm> "A";"B"^{\alpha_1}
\ar@ <1mm>  "B";"C"^{\alpha_2}
\ar@ <1mm>  "C";"D"^{\alpha_3}
\ar@ <1mm> "B";"A"^{\beta_1}
\ar@ <1mm>  "C";"B"^{\beta_2}
\ar@ <1mm>  "D";"C"^{\beta_3}
\end{xy}
\]
and let $I$ be the admissible ideal of $KQ$ which defines the relations
\begin{gather*}
\alpha_1\alpha_2\alpha_3=0,\;\; \beta_3\beta_2\beta_1=0 \\
\beta_1\alpha_1\alpha_2=\alpha_2\alpha_3\beta_3, \;\; \beta_2\beta_1\alpha_1=\alpha_3\beta_3\beta_2 \\
\alpha_1\beta_1\alpha_1=\alpha_1\alpha_2\beta_2, \;\; \beta_1\alpha_1\beta_1=\alpha_2\beta_2\beta_1 \\
\alpha_2\beta_2\alpha_2=0=\beta_2\alpha_2\beta_2 \\
\alpha_3\beta_3\alpha_3=\beta_2\alpha_2\alpha_3, \;\; \beta_3\alpha_3\beta_3=\beta_3\beta_2\alpha_2
\end{gather*}
and all the path of length greater than or equal to $5$ are set to be zero. 

If $\lambda=(-1)^{\ell+1}=-1$ and $K$ is of odd characteristic, then we have the algebra isomorphism $R^\Lambda(\delta)\simeq KQ/I$. 
\end{lem}
\begin{proof}
Let $\alpha_1=e_1\psi_2e_2$, $\alpha_2=e_2\psi_1e_3$, $\alpha_3=e_3\psi_2e_4$ and $\beta_1=e_2\psi_1e_1$, $\beta_2=e_3\psi_1e_2$, $\beta_3=e_4\psi_2e_3$. 
Then, we can check that the defining relations are satisified and prove the desired isomorphism. The details are left to the reader. 
\end{proof}

Let $A=KQ/I$ be the bounded quiver algebra in Lemma \ref{quiver presentation}, $P_i=R^\Lambda(\delta)e_i$, for $1\le i\le 4$. 
We shall mutate the stalk complex $P_1\oplus P_2\oplus P_3\oplus P_4$ at $P_i$, for $1\le i\le 4$. As $A$ admits 
an algebra automorphism of order $2$ which swap
\[
\alpha_1\leftrightarrow \beta_3,\;\; \alpha_2\leftrightarrow \beta_2,\;\; \alpha_3\leftrightarrow \beta_1,
\]
it suffices to consider mutations at $P_1$ and $P_2$. 

Let us start with mutation of $P_1\oplus P_2\oplus P_3\oplus P_4$ at $P_1$. Since
\begin{align*}
P_1&=\Span\{ e_1, \beta_1, \alpha_1\beta_1, \beta_2\beta_1, \alpha_2\beta_2\beta_1, \alpha_1\alpha_2\beta_2\beta_1\},\\
P_2&=\Span\{ e_2, \underline{\alpha_1}, \beta_2, \alpha_2\beta_2, \beta_1\alpha_1, \beta_3\beta_2, \underline{\alpha_1\alpha_2\beta_2}, \beta_2\beta_1\alpha_1, \alpha_2\alpha_3\beta_3\beta_2\},\\
P_3&=\Span\{ e_3, \alpha_2, \beta_3, \underline{\alpha_1\alpha_2}, \beta_2\alpha_2, \alpha_3\beta_3, \alpha_2\alpha_3\beta_3, \beta_3\beta_2\alpha_2, \alpha_3\beta_3\beta_2\alpha_2\},\\
P_4&=\Span\{ e_4, \alpha_3, \alpha_2\alpha_3, \beta_3\alpha_3, \beta_2\alpha_2\alpha_3, \beta_3\beta_2\alpha_2\alpha_3\},
\end{align*}
where the underlined elements may become a target of a morphism $P_1\to P_i$, for $i=2,3,4$, 
the right multiplication by $\alpha_1$ gives the minimal left ${\rm add}(P_2\oplus P_3\oplus P_4)$-approximation $P_1\to P_2$. Therefore, the complex 
$Q=(P_1\stackrel{\cdot \alpha_1}{\rightarrow} P_2)$ concentrated in degrees $-1$ and $0$ is the mapping cone. The next proposition shows that 
the mutated algebra cannot be a block algebra of Hecke algebras of classical type. 

\begin{prop}
The algebra $\End_{K^b(proj(A))}(Q\oplus P_2\oplus P_3\oplus P_4)^{\rm op}$ is not cellular.
\end{prop}
\begin{proof}
Let $\alpha_1'\in \Hom_{K^b(proj(A))}(Q,P_2)$ and $\beta_1'\in \Hom_{K^b(proj(A))}(P_2,Q)$ be 
\[
\begin{array}{ccc}
P_1 & \stackrel{\cdot\alpha_1}{\longrightarrow} & P_2 \\
\downarrow & & \downarrow  \\
0 & \longrightarrow & P_2 
\end{array}
\quad\text{and}\quad
\begin{array}{ccc}
0 & \longrightarrow & P_2 \\
\downarrow & & \downarrow \\
P_1 & \stackrel{\cdot\alpha_1}{\longrightarrow} & P_2
\end{array}
\]
respectively, where the right vertical homomorphism is the right multiplication by $\beta_1\alpha_1-\alpha_2\beta_2$ for $\alpha_1'$ and the identity map for $\beta_1'$, 
and $\gamma\in \Hom_{K^b(proj(A))}(Q,P_4)$
\[
\begin{array}{ccc}
P_1 & \stackrel{\cdot\alpha_1}{\longrightarrow} & P_2 \\
\downarrow & & \downarrow  \\
0 & \longrightarrow & P_4 
\end{array}
\]
where the right vertical homomorphism is the right multiplication by $\alpha_2\alpha_3$. 
Then we can give the bounded quiver algebra presentation of the algebra
\[
B=\End_{K^b(proj(A))}(Q\oplus P_2\oplus P_3\oplus P_4)^{\rm op}. 
\]
Namely, after computing homomorphisms between $Q$ and $P_2,P_3,P_4$, 
the Gabriel quiver of $B$ is of the form
\[
\begin{xy}
(20,20) *[o]+{1}="A", (20,0) *[o]+{2}="B", (40,0) *[o]+{3}="C", (40,20) *[o]+{4}="D",
\ar@ <1mm> "A";"B"^{\alpha_1'}
\ar@ <1mm>  "B";"C"^{\alpha_2}
\ar@ <1mm>  "C";"D"^{\alpha_3}
\ar@ <1mm> "B";"A"^{\beta_1'}
\ar@ <1mm>  "C";"B"^{\beta_2}
\ar@ <1mm>  "D";"C"^{\beta_3}
\ar@ <0mm>  "A";"D"^{\gamma}
\end{xy}
\]
so that $B$ is not cellular since there does not exist an arrow $4\to 1$. 
\end{proof}

\begin{rmk}
We denote the indecomposable projective $B$-modules by $P_1', P_2', P_3', P_4'$. Then, the radical series and the socle series coincide for each $P_i'$. 
\begin{itemize}
\item[(1)]
$\Soc(P_1')=K \alpha_1'\alpha_2\beta_2\beta_1'$ and $\Rad(P_1')/\Soc(P_1')$ is of length $3$ as follows. 
\[
\begin{array}{c}
K\beta_1' \\
K\alpha_1'\beta_1'\oplus K\beta_2\beta_1'\\
K\alpha_2\beta_2\beta_1'\oplus K\beta_3\beta_2\beta_1'
\end{array}
\]

\noindent
where $\alpha_2\beta_2\beta_1'+\beta_1'\alpha_1'\beta_1'=0$, $\beta_3\beta_2\beta_1'\alpha_1'=0$, $\alpha_3\beta_3\beta_2\beta_1'=0$ and $\gamma\beta_3=\alpha_1'\alpha_2$. 
\item[(2)]
$\Soc(P_2')=K \alpha_2\alpha_3\beta_3\beta_2$ and $\Rad(P_2')/\Soc(P_2')$ is of length $3$ as follows. 
\[
\begin{array}{c}
K\alpha_1'\oplus K\beta_2 \\
K\beta_1'\alpha_1'\oplus K\alpha_2\beta_2\oplus K\beta_3\beta_2 \\
K\alpha_1'\alpha_2\beta_2\oplus K\beta_2\beta_1'\alpha_1'
\end{array}
\]

\noindent
where $\alpha_1'\beta_1'\alpha_1'+\alpha_1'\alpha_2\beta_2=0$, $\beta_2\beta_1'\alpha_1'=\alpha_3\beta_3\beta_2$ and $\beta_1'\alpha_1'\alpha_2\beta_2=\alpha_2\alpha_3\beta_3\beta_2$. 
\item[(3)]
$\Soc(P_3')=K \alpha_3\beta_3\beta_2\alpha_2$ and $\Rad(P_3')/\Soc(P_3')$ is of length $3$ as follows. 
\[
\begin{array}{c}
K\alpha_2\oplus K\beta_3 \\
K\alpha_1'\alpha_2\oplus K\beta_2\alpha_2\oplus K\alpha_3\beta_3 \\
K\beta_1'\alpha_1'\alpha_2\oplus K\beta_3\beta_2\alpha_2
\end{array}
\]

\noindent
where $\beta_1'\alpha_1'\alpha_2=\alpha_2\alpha_3\beta_3$ and $\beta_2\beta_1'\alpha_1'\alpha_2=\alpha_3\beta_3\beta_2\alpha_2$. 
\item[(4)]
$\Soc(P_4')=K \beta_3\beta_2\alpha_2\alpha_3$ and $\Rad(P_4')/\Soc(P_4')$ is of length $3$ as follows. 
\[
\begin{array}{c}
K\alpha_3\oplus K\gamma\\
K\alpha_2\alpha_3\oplus K\beta_3\alpha_3\\
K\beta_2\alpha_2\alpha_3
\end{array}
\]

\noindent
where $\beta_1'\gamma=\alpha_2\alpha_3$ and $\alpha_1'\alpha_2\alpha_3=0$. 
\end{itemize}
\end{rmk}

If we consider the mutation of the stalk complex $P_1\oplus P_2\oplus P_3\oplus P_4$ at $P_2$, the minimal left ${\rm add}(P_1\oplus P_3\oplus P_4)$-approximation is 
$P_2\rightarrow P_1\oplus P_3$ given by the right multiplication of $(\beta_1,\alpha_2)$, and we define the mapping cone to be $R$. 

We define $\alpha_1' \in \Hom_{K^b(proj(A))}(P_1,R)$, $\alpha_2' ,\gamma \in \Hom_{K^b(proj(A))}(R,P_3)$ by
\[
\begin{array}{ccc}
0 & \longrightarrow & P_1 \\
\downarrow & & \downarrow  \\
P_2 & \stackrel{\cdot(\beta_1,\alpha_2)}{\longrightarrow} & P_1\oplus P_3
\end{array}
\quad\text{and}\quad
\begin{array}{ccc}
P_2 & \stackrel{\cdot(\beta_1,\alpha_2)}{\longrightarrow} & P_1\oplus P_3 \\
\downarrow & & \downarrow \\
0 & \longrightarrow & P_3
\end{array}
\]
respectively, where the right vertical homomorphism is the right multiplication by $(e_1,0)$ for $\alpha_1'$ and 
$(\begin{smallmatrix} \alpha_1\alpha_2 \\ -\alpha_3\beta_3 \end{smallmatrix})$ for $\alpha_2'$, 
$(\begin{smallmatrix} 0 \\ -\beta_2\alpha_2 \end{smallmatrix})$ for $\gamma$. 

Similarly, we define $\beta_1' \in \Hom_{K^b(proj(A))}(R,P_1)$, $\beta_2' \in \Hom_{K^b(proj(A))}(P_3,R)$ by
\[
\begin{array}{ccc}
P_2 & \stackrel{\cdot(\beta_1, \alpha_2)}{\longrightarrow} & P_1\oplus P_3 \\
\downarrow & & \downarrow  \\
0 & \longrightarrow & P_1
\end{array}
\quad\text{and}\quad
\begin{array}{ccc}
0 & \longrightarrow & P_3 \\
\downarrow & & \downarrow \\
P_2 & \stackrel{\cdot(\beta_1,\alpha_2)}{\longrightarrow} & P_1\oplus P_3
\end{array}
\]
respectively, where the right vertical homomorphism is the right multiplication by 
$(\begin{smallmatrix} \alpha_1\beta_1 \\ -\beta_2\beta_1 \end{smallmatrix})$ for $\beta_1'$ and $(0, e_3)$ for $\beta_2'$. 
Let $C=\End_{K^b(proj(A))}(P_1\oplus R\oplus P_3\oplus P_4)^{\rm op}$. Then we can show that the Gabriel quiver is as follows. 
\[
\begin{xy}
(0,0) *[o]+{1}="A", (20,0) *[o]+{2}="B", (40,0) *[o]+{3}="C", (60,0) *[o]+{4}="D",
\ar@ <1mm> "A";"B"^{\alpha_1'}
\ar@ <1mm>  "B";"C"^{\alpha_2'}
\ar@ <1mm>  "C";"D"^{\alpha_3}
\ar@ <1mm> "B";"A"^{\beta_1'}
\ar@ <1mm>  "C";"B"^{\beta_2'}
\ar@ <1mm>  "D";"C"^{\beta_3}
\ar@(lu,ru)  "B";"C"^{\gamma}
\end{xy}
\]
Hence, we have the following. 

\begin{prop}
The algebra $\End_{K^b(proj(A))}(P_1\oplus R\oplus P_3\oplus P_4)^{\rm op}$ is not cellular.
\end{prop}

\begin{rmk}
We denote the indecomposable projective $C$-modules by $P_1', P_2', P_3', P_4'$ as before. Then, the module structure this time are as follows. 
\begin{itemize}
\item[(1)]
$\Soc(P_1')=K\alpha_1'\beta_1'\alpha_1'\beta_1'$ and $\Rad(P_1')/\Soc(P_1')$ is of length $3$ as follows. 
\[
\begin{array}{c}
K\beta_1' \\
K\alpha_1'\beta_1'\oplus K\beta_2'\beta_1'\\
K\beta_1'\alpha_1'\beta_1'
\end{array}
\]

\noindent
where $\gamma\beta_2'\beta_1'=0$, $\alpha_2'\beta_2'\beta_1'+\beta_1'\alpha_1'\beta_1'=0$, $\beta_3\beta_2'\beta_1'=0$ and $\beta_2'\beta_1'\alpha_1'\beta_1'=0$. 
\item[(2)]
$\Soc(P_2')= K\beta_1'\alpha_1'\beta_1'\alpha_1'$ and $\Rad(P_2')/\Soc(P_2')$ is of length $3$ as follows.
\[
\begin{array}{c}
K\alpha_1'\oplus K\beta_2' \\
K\beta_1'\alpha_1'\oplus K\alpha_2'\beta_2'\oplus K\gamma\beta_2'\oplus K\beta_3\beta_2' \\
K\beta_2'\gamma\beta_2'\oplus K\alpha_1'\beta_1'\alpha_1' \oplus K\beta_2'\alpha_2'\beta_2'
\end{array}
\]

\noindent
where $\alpha_1'\beta_1'\alpha_1'+\alpha_1'\alpha_2'\beta_2'=0$, $\beta_2'\alpha_2'\beta_2'+\alpha_3\beta_3\beta_2'=0$, 
$\beta_2'\beta_1'\alpha_1'+\beta_2'\gamma\beta_2'=0$ and $\alpha_1'\gamma=0$, 
$\alpha_2'\beta_2'\gamma\beta_2'=\beta_1'\alpha_1'\beta_1'\alpha_1'=\gamma\beta_2'\gamma\beta_2'$, 
$\alpha_2'\beta_2'\alpha_2'\beta_2'=\beta_1'\alpha_1'\beta_1'\alpha_1'=\gamma\beta_2'\alpha_2'\beta_2'$, 
$\beta_3\beta_2'\gamma\beta_2'=0$, $\beta_3\beta_2'\alpha_2'\beta_2'=0$. 
\item[(3)]
$\Soc(P_3')=K\beta_2'\gamma\alpha_3\beta_3$ and $\Rad(P_3')/\Soc(P_3')$ is of length $3$ as follows. 
\[
\begin{array}{c}
K\alpha_2'\oplus K\gamma\oplus K\beta_3 \\
K\alpha_1'\alpha_2'\oplus K\beta_2'\gamma\oplus K\beta_2'\alpha_2' \\
K\alpha_2'\beta_2'\alpha_2'\oplus K\beta_3\beta_2'\alpha_2'
\end{array}
\]

\noindent
where $\beta_2'\alpha_2'+\alpha_3\beta_3=0$, $\beta_1'\alpha_1'\alpha_2'+\alpha_2'\beta_2'\alpha_2'=0$, 
$\alpha_2'\beta_2'\gamma=\alpha_2'\beta_2'\alpha_2'=\gamma\beta_2'\alpha_2'$, 
$\beta_3\beta_2'\gamma=\beta_3\beta_2'\alpha_2'$, $\gamma\beta_2'\gamma=0$, 
$\beta_2'\gamma\alpha_3\beta_3+\beta_2'\alpha_2'\beta_2'\alpha_2'=0$, $\beta_2'\gamma\alpha_3\beta_3=\alpha_3\beta_3\beta_2'\alpha_2'$ and 
$\alpha_1'\alpha_2'\beta_2'\alpha_2'=0$. 
\item[(4)]
$\Soc(P_4')=K \beta_3\beta_2'\alpha_2'\alpha_3$ and $\Rad(P_4')/\Soc(P_4')$ is of length $3$ as follows. 
\[
\begin{array}{c}
K\alpha_3 \\
K\alpha_2'\alpha_3\oplus K\beta_3\alpha_3\\
K\beta_2'\alpha_2'\alpha_3
\end{array}
\]

\noindent
where $\gamma\alpha_3+\alpha_2'\alpha_3=0$, $\beta_2'\alpha_2'\alpha_3+\alpha_3\beta_3\alpha_3=0$, $\alpha_1'\alpha_2'\alpha_3=0$, 
$\alpha_2'\beta_2'\alpha_2'\alpha_3=0$ and $\gamma\beta_2'\alpha_2'\alpha_3=0$. 
\end{itemize}
\end{rmk}

It is an interesting question to ask whether there exists a symmetric cellular algebra which is not a block algebra in the derived equivalence class of $R^\Lambda(\delta)$. 

\bibliographystyle{amsplain}

\end{document}